\documentclass[letterpaper,11pt]{amsart}

\usepackage{amsfonts,amsthm,latexsym,amsmath,amssymb,amscd,amsmath, mathrsfs, epsf, geometry}
\usepackage{color}
\usepackage[utf8]{inputenc}
\usepackage{picture}
\usepackage{tikz}
\usepackage{tikz-cd}
\usepackage{graphicx}
\usepackage{verbatim}
\usepackage{hyperref}
\hypersetup{
    colorlinks,
    linkcolor={red!50!black},
    citecolor={green!50!black},
    urlcolor={blue!80!black}
}

\newtheorem{theorem}{Theorem}[section]

\newtheorem{lemma}[theorem]{Lemma}
\newtheorem{proposition}[theorem]{Proposition}
\newtheorem{conj}{Conjecture}
\theoremstyle{definition}
\newtheorem{definition}[theorem]{Definition}
\newtheorem{example}[theorem]{Example}
\newtheorem{question}[theorem]{Question}

\theoremstyle{remark}
\newtheorem{remark}[theorem]{Remark}

\newcommand{\aug}{\fboxsep=-\fboxrule\!\!\!\fbox{\strut}\!\!\!}

\let\phi\relax\DeclareMathOperator{\phi}{\varphi}
\DeclareMathOperator{\osec}{\overset{\circ}{\sigma}}

\DeclareMathOperator{\NN}{\mathbb{N}}

\DeclareMathOperator{\CC}{\mathbb{C}}

\DeclareMathOperator{\PP}{\mathbb{P}}
\DeclareMathOperator{\KK}{\mathbb{K}}
\DeclareMathOperator{\QQ}{\mathbb{Q}}
\DeclareMathOperator{\ZZ}{\mathbb{Z}}
\DeclareMathOperator{\FF}{\mathbb{F}}
\DeclareMathOperator{\TT}{\mathbb{T}}

\DeclareMathOperator{\GL}{GL}
\DeclareMathOperator{\SL}{SL}

\let\sl\relax\DeclareMathOperator{\sl}{\mathfrak{sl}}
\DeclareMathOperator{\gl}{\mathfrak{gl}}

\DeclareMathOperator{\rk}{rk}
\DeclareMathOperator{\mrk}{mrk}
\DeclareMathOperator{\mbrk}{\underline{\mrk}}
\DeclareMathOperator{\im}{im}
\let\span\relax\DeclareMathOperator{\span}{span}

\DeclareMathOperator{\Sym}{Sym}

\DeclareMathOperator{\tr}{tr}

\DeclareMathOperator{\la}{\langle}
\DeclareMathOperator{\ra}{\rangle}
\newcommand{\ol}[1]{\overline{#1}}

\title{The monic rank}

\author{Arthur Bik}
\address{University of Bern, Mathematical Institute, Alpeneggstrasse 22,
3012 Bern, Switzerland}
\email{arthur.bik@math.unibe.ch}

\author{Jan Draisma}
\address{University of Bern, Mathematical Institute, Sidlerstrasse 5,
3012 Bern, Switzerland, and Eindhoven University of Technology}
\email{jan.draisma@math.unibe.ch}

\author{Alessandro Oneto}
\address{Barcelona Graduate School of Mathematics, and Universitat Polit\`ecnica de Catalunya, Av.da Diagonal 647 (ETSEIB), 08028 Barcelona, Spain}
\email{alessandro.oneto@upc.edu, aless.oneto@gmail.com}

\author{Emanuele Ventura}
\address{Department of Mathematics, Texas A\&M University, College Station, TX 77843-3368, USA}
\email{eventura@math.tamu.edu, emanueleventura.sw@gmail.com}

\thanks{AB was supported by JD's Vici grant. JD was partially supported by the NWO Vici grant entitled {\em Stabilisation in Algebra and Geometry}. AO acknowledges financial support from the Spanish Ministry of Economy and Competitiveness, through the Mar\'ia de Maeztu Programme for Units of Excellence in R$\&$D (MDM-2014-0445). EV acknowledges financial support by the grant 346300 for IMPAN from the Simons Foundation and the matching 2015-2019 Polish MNiSW fund. AO and EV thank the University of Bern (Switzerland) for the hospitality during visits where the project of the paper was discussed.}

\date{}

\begin{document}

\begin{abstract}
We introduce the {\em monic rank} of a vector relative to an
affine-hyperplane section of an irreducible Zariski-closed affine cone $X$. We show
that the monic rank is finite and greater than or equal to the
usual $X$-rank. We describe an algorithmic technique based on classical
invariant theory to determine, in concrete situations, the maximal
monic rank. Using this technique, we establish three new instances of
a conjecture due to B.~Shapiro which states that a binary form of
degree $d\cdot e$ is the sum of $d$ $d$-th powers of forms of
degree $e$. Furthermore, in the case where $X$ is the cone of
highest weight vectors in an irreducible representation---this includes
the well-known cases of tensor rank and symmetric rank---we raise the
question whether the maximal rank {\em equals} the maximal monic rank.
We answer this question affirmatively in several instances.
\end{abstract}

\maketitle


\section{Introduction}\label{sec:intro}

Let $\KK$ be an algebraically closed field of characteristic
zero. All our vector spaces and algebraic varieties will be
over $\KK$, finite-dimensional, reduced and identified with their sets of $\KK$-points. 

\subsection*{Monic secant varieties and monic rank}
Let $V$ be a finite-dimensional vector space.  Let $X\subseteq V$
be an irreducible Zariski-closed affine cone such that its $\KK$-linear span 
equals~$V$, i.e., the cone $X$ is non-degenerate. 
A very fruitful field of research investigates the problem of minimally decomposing an element $v \in V$ as a sum of points on $X$. Following well-established terminology in recent literature, we call the minimal number of points for which this is possible the \textit{$X$-rank} of $v$ and we denote it by $\rk_X(v)$. We refer to \cite{LandsbergBook, FourLectures, HitchhikerGuide} and their references for an exhaustive exposition of the problem. In the cases where $V$ is, for example, some vector space of tensors and $X$ is the subvariety of rank~$\leq 1$ tensors in~$V$, the study of $X$-ranks has very interesting relations with fields in applied mathematics.
 In this paper we introduce a new, but related, type of rank.\bigskip

Let $h\in V^{*}\setminus\{0\}$ 
be a non-zero linear function and consider the affine hyperplane
$$
H = \{v\in V \mid h(v) = 1\}\subseteq V.
$$
We write $X_1$ for the affine-hyperplane section $X\cap H$ of $X$.

\begin{definition}
Let $k$ be a positive integer. The {\em $k$-th open secant variety} of $X_1$ is the set 
$$
\osec_k X_1 := \{ p_1+\cdots+p_k \mid p_1,\dots,p_k\in X_1 \}. 
$$
This is a subset of $kH := \{kp \mid p \in H\}$ where $kp$ is the sum of $k$ copies of $p$. We define $\sigma_k X_1$ to be the Zariski
closure of $\osec_k X_1$ and call this set the 
{\em $k$-th secant variety} of $X_1$. We also call $\osec_k X_1$ the 
{\em $k$-th open monic secant variety} of $X$ and $\sigma_k X_1$ the 
{\em $k$-th monic secant variety} of $X$. 
\end{definition}

Since $H$ is an affine space, we have $kH = \osec_k H =  \sigma_k H$. And, for any $X$, we have 
$$
k X_1 \subseteq \osec_k X_1 \subseteq \sigma_kX_1.
$$
However as we shall see, both inclusions can be strict. We now
define the monic rank of a vector $v\in V$ with $h(v)\neq0$.

\begin{definition}
Let $v\in V\setminus \ker(h)$ be a vector. The {\it monic rank} of $v$ is defined to be
$$
\mrk_{X,h}(v) := \inf\left\{ k\in\ZZ_{\geq1} ~\middle|~ \frac{k}{h(v)}\cdot v\in \osec_kX_1 \right\}.
$$
Similarly, the {\em monic border rank} of $v$ is
$$
\mbrk_{X,h}(v) :=\inf \left\{ k\in\ZZ_{\geq1} ~\middle|~ \frac{k}{h(v)}\cdot v \in \sigma_kX_1\right\}.
$$
\end{definition}

The following example is illustrative for the rest of our results. 

\begin{example}
Let $V=\KK[x,y]_{(2)}$ be the vector space of binary forms of degree $2$ and let $X \subseteq V$ be the subset of squares of linear forms. We consider the linear function $h\in V^*\setminus\{0\}$ which selects the coefficient of $x^2$. So
$$
h\colon V \rightarrow \KK, \quad ax^2+bxy+cy^2 \mapsto a.
$$
We get $H=\{x^2 + bxy +cy^2 \mid b,c \in \KK\}$ and $X_1=X\cap H=\{(x+a)^2 \mid a \in \KK\}$.
Now, an element of $V$ is contained in the second open monic secant $\osec_2 X_1$ if and only if it equals 
$$
(x+a_1)^2 + (x+a_2)^2=2x^2 + 2(a_1+a_2)x+(a_1^2+a_2^2)
$$
for some $a_1,a_2\in\KK$. Notice that the polynomials $2(a_1+a_2)$ and $a_1^2+a_2^2$ generate the ring
of symmetric polynomials in the variables $a_1$ and $a_2$, i.e., the invariant ring $\KK[a_1,a_2]^{\mathfrak S_2}$.  Here $\mathfrak S_2$ is the symmetric group on two letters and acts by permuting $a_1$ and $a_2$. From classical invariant theory, we know that the map 
\[ 
\KK^2 \to \KK^2, \quad (a_1,a_2) \mapsto (2(a_1+a_2),a_1^2+a_2^2)
\]
is a finite morphism. Thus, it is closed and dominant and so it is also surjective. Hence 
$$
\osec_2 X_1=\sigma_2 X_1= 2H.
$$
See Proposition \ref{prop:Closed} and the proof of Theorem \ref{thm:Shapiro} for an explanation. We find that any $v \in V \setminus\ker(h)$ satisfies $\mrk_{X,h}(v) \leq 2$.
 \hfill $\clubsuit$
\end{example}

\subsection*{Main results}

{\em A priori} it is not clear that the monic rank of an element of $V\setminus\ker(h)$ is even finite. This is our first foundational result.

\begin{theorem} \label{thm:Basic}
The function $k \mapsto \dim \sigma_k X_1$
is strictly increasing until it coincides with its maximal value $\dim H=\dim V-1$ and constant from then on.
Let $k_0$ be the minimal $k$ integer for which $\sigma_k X_1=kH$ holds. 
Then for any $v \in V \setminus \ker(h)$, we have
\[ \rk_X(v) \leq \mrk_{X,h}(v) \leq 2k_0.\]
In particular, the monic rank is finite.
\end{theorem}

\begin{definition}
The minimal integer $k$ for which
$\sigma_kX_1 = kH$ is called the {\it generic monic rank} of elements of $V \setminus \ker(h)$.  
\end{definition}

Theorem \ref{thm:Basic} mimics a result of Blekherman and
Teitler \cite[Theorem 3]{Blekherman15} that relates the maximal
$X$-rank to the generic $X$-rank. Another remarkable general result on maximal $X$-rank is \cite[Proposition 5.1]{LanTei10}. Extensive literature has been devoted to the particular case of finding the maximal rank in the space of homogeneous polynomials of given degree and given number of variables with respect to the space of pure powers of linear forms. See for example \cite{Seg42, Kle99, CS11, BGI11, Jel14, DP15a, DP15b, BT16}.

\medskip

Our main motivations for introducing and studying monic ranks are questions concerning the maximal ranks of elements of particular varieties, in which the maximal rank turns out to be bounded by the maximal monic rank. The first instance of such a question comes from an interesting conjecture due to B. Shapiro (see \cite[Conjecture 1.4]{LORS}) which states that every binary form of degree $d\cdot e$ is the sum of $d$ $d$-th powers of forms of degree~$e$. We show that, in a few particular cases, a stronger statement holds: every binary form of degree $d\cdot e$ whose coefficient at $x^{de}$ equals $d$ is the sum of $d$ $d$-th powers of {\em monic} forms of degree $e$. 

\begin{theorem} \label{thm:Shapiro}
Let $d$ and $e$ be positive integers. Let $V=\KK[x,y]_{(de)}$ be the space
of binary forms of degree $de$, take $X=\{f^d \mid f \in \KK[x,y]_{(e)}\}$ 
and let $h$ be the linear function that maps a form $q \in
\KK[x,y]_{(de)}$ to its coefficient at $x^{de}$. Then, in the following
cases, the maximal monic rank is at most $d$:
\begin{enumerate}
\item[(i)] $e=1$ and arbitrary $d$;
\item[(ii)] $d\in\{1,2\}$ and arbitrary $e$; 
\item[(iii)] $d=3$ and $e\in\{2,3,4\}$;
\item[(iv)] $d=4$ and $e=2$.
\end{enumerate}
In particular, in these cases, the $X$-rank of any $q \in V$ is at most $d$.
\end{theorem} 

In terms of the (non-monic) Shapiro's conjecture on writing binary forms of degree $de$ as sums of $d$-th powers:
\begin{itemize}
\item the case $e = 1$ is classical (see, e.g., \cite[Theorem 4.9]{RezLength}); 
\item the case $d = 1$ is trivial;
\item the case $d = 2$ is quite immediate (see \cite[Theorem 5]{FOS}); and
\item the case $(d,e)=(3,2)$ was proven in \cite[Theorem 3.1]{LORS}.
\end{itemize}
As far as we know, the cases where $(d,e) \in \{(3,3),(3,4),(4,2)\}$ are new.\bigskip

We next turn our attention to a particularly nice class of varieties.
Let $G$ be a connected reductive algebraic group over $\KK$ and let $V$
be an irreducible finite-dimensional rational representation of $G$. Let $X \subseteq V$
be the cone of highest-weight vectors, i.e., the affine cone over the
unique closed $G$-orbit in $\PP(V)$. The latter projective variety is called
a {\it homogeneous variety}. Let $h$ be a highest weight vector
in the dual $G$-module $V^*$. 

\begin{question} \label{que:Minorbit}
In the setting above, is the maximal $X$-rank of a vector in
$V$ equal to the maximal {\em monic} $X$-rank of a vector in $V\setminus\ker(h)$?
\end{question}

Let $V$ be a finite-dimensional vector space. Then we denote the $n$-th symmetric power of~$V$ by $\Sym^n(V)$.

\begin{theorem}\label{thm:Minorbit}
The answer to Question~\ref{que:Minorbit} is affirmative in the following instances:
\begin{enumerate}
\item[(i)] $G = \GL_2$ and $V=\Sym^d(\KK^2) \cong\KK[x,y]_{(d)}$;
\item[(ii)] $G=\GL_m \times \GL_n$ and $V=\KK^{m\times n}$;
\item[(iii)] $G=\GL_n$ and $V= \Sym^2(\KK^n) \cong
\KK[x_1,\dots,x_n]_{(2)} \cong \{A\in\KK^{n\times n}\mid A=A^T\}$; 
\item[(iv)] $G=\GL_2 \times \GL_2 \times \GL_2$ and $V=\KK^2 \otimes \KK^2
\otimes \KK^2$; and
\item[(v)] $G=\SL_n$ and $V=\sl_n$ is its adjoint representation.
\end{enumerate}
\end{theorem}

\begin{remark}
The first case of the theorem corresponds to the Waring rank of
binary forms, which is also (i) in Theorem~\ref{thm:Shapiro}. The second case corresponds to the
usual matrix rank of $m\times n$ matrices. Case (iii)
corresponds to the symmetric matrix rank of symmetric $n \times
n$ matrices (or to the Waring rank of quadrics in $n$
variables). Case (iv) corresponds to the tensor rank of $2
\times 2 \times 2$ tensors. And lastly, case (v) corresponds
to the $X$-rank of trace-zero $n\times n$ matrices with $X$ being the affine cone over the
projective adjoint variety of incident
point-hyperplane pairs in $\PP^{n-1}$.
\end{remark}

Admittedly, this is not much evidence for an affirmative answer to
Question~\ref{que:Minorbit} in general. Moreover, our proofs in each of these cases
are {\em ad hoc}. We therefore appeal to the reader for 
different approaches to Question~\ref{que:Minorbit}. In
particular, an affirmative answer to that question for $G=\GL_n$ and
$V=\KK[x_1,\ldots,x_n]_{(d)}$ would imply that a new lower bound~of
$$
\left\lceil\frac{\dim V-1}{\dim X-1}\right\rceil=\left\lceil\frac{1}{n-1}\left(\binom{d+n-1}{d}-1 \right) \right\rceil 
$$
on the maximal rank of a form of degree $d$. This is
almost always bigger than 
$$
\left\lceil\frac{1}{n}\binom{d+n-1}{d} \right\rceil, 
$$
which (except for quadrics and finitely many further exceptions)
is the generic Waring rank by the Alexander-Hirschowitz theorem
\cite{Alexander95}. This would give a positive answer to the open question of whether the maximal rank is always bigger than the generic rank when we have $d \geq 3$.

\subsection*{Structure of the paper}
In Section \ref{sec:Basics}, we lay the foundations of the notion of monic rank and prove Theorem~\ref{thm:Basic}. In Section \ref{sec:Invariant}, we develop some machinery from classical invariant
theory to compute the maximal monic rank in certain explicit cases. We
apply this machinery to the proof of Theorem~\ref{thm:Shapiro} in
Section \ref{sec:Shapiro}. Finally, in Section \ref{sec:Minorbit}, we
establish Theorem~\ref{thm:Minorbit}.

\section{The basics of monic rank}\label{sec:Basics}

Recall the following notation from the introduction: 
\begin{itemize}
	\item $X \subseteq V$ denotes a non-degenerate irreducible affine cone;
	\item $h \in V^*\setminus\{0\}$ is a non-zero linear function; and
	\item $H = h^{-1}(1)$ is an affine hyperplane and $X_1 := X\cap H$.
\end{itemize} 
Theorem~\ref{thm:Basic} will be a direct consequence of Proposition \ref{prop:mongenexists} and Proposition \ref{prop:mrank-welldef}.

\begin{proposition} \label{prop:mongenexists}
The function $k \mapsto \dim \sigma_k X_1$ is strictly increasing until
it coincides with its maximal value, which is $\dim H=\dim V-1$, and constant from then on.  Consequently, the function $\mbrk_{X,h}$ is bounded. Moreover, its value at $v$ is an upper bound on the ordinary border $X$-rank of $v$ for all vectors $v\in V\setminus\ker(h)$.
\end{proposition}
\begin{proof}
Let $p_1\in X_1$ be any point. Then $\sigma_kX_1+p_1\subseteq \sigma_{k+1} X$. So the function $k \mapsto
\dim \sigma_k X_k$ is weakly increasing. Since $\dim \sigma_k X_1$ is bounded from above by $\dim H$, there exists a $k$ such that $\dim \sigma_kX_1=\dim \sigma_{k+1} X_1$. Let $k$ be any positive integer with this property. Then, since both $\sigma_kX_1$ and $\sigma_{k+1}X_1$ are irreducible, the isomorphism 
\begin{eqnarray*}
kH &\to& (k+1)H\\
v &\mapsto& v+p_1
\end{eqnarray*}
restricts to an isomorphism between $\sigma_k X_1$ and $\sigma_{k+1} X_1$.
By definition, a general point on $\sigma_{k+2} X_1$ is of the form $v+p_2$ with
$v \in \sigma_{k+1} X_1$ and $p_2 \in X_1$ and we have
$$
v+p_2=(v-p_1)+p_2+p_1 \in \sigma_k X_1 + X_1 +p_1 \subseteq \sigma_{k+1} X_1+p_1.
$$
Therefore, the isomorphism
\begin{eqnarray*}
(k+1)H &\to& (k+2)H\\
w &\mapsto& w+p_1
\end{eqnarray*}
restricts to an isomorphism between $\sigma_{k+1} X$ and $\sigma_{k+2} X$. Now, let $k_0$ be the minimal value of $k$ for which $\dim \sigma_k X_1=\dim \sigma_{k+1}X_1$. Then we conclude that the function $k \mapsto \dim \sigma_kX_1$ is strictly increasing for $k < k_0$ and is constant for $k \geq k_0$.\bigskip

Next, we show that $\sigma_{k_0}X_1=k_0H$, which implies in particular that $\dim\sigma_{k_0}X_1=\dim H$. Let $a,b$ be positive integers. Then 
$$
(a+b) \sigma_{k_0} X_1 \subseteq a \sigma_{k_0} X_1 + b \sigma_{k_0 X_1} \subseteq \sigma_{(a+b)k_0} X_1.
$$
Since the leftmost and rightmost sets are closed, irreducible and of the same
dimension, all three sets coincide. For any $p_1,p_2 \in
\sigma_{k_0} X_1$, we have $a p_1 + b p_2 \in \sigma_{(a+b)k_0} X_1 = (a+b)\sigma_{k_0} X_1$.
Therefore
$$
\frac{a}{a+b} \cdot p_1 + \frac{b}{a+b}\cdot p_2 \in \sigma_{k_0} X_1.
$$
So the line through $p_1$ and $p_2$ intersects $\sigma_{k_0} X_1$ in infinitely many points. Hence this line must be entirely contained in $\sigma_{k_0} X_1$. Since this holds for all $p_1,p_2\in\sigma_{k_0}X_1$, we see that $\sigma_{k_0} X_1$ is an affine space. Since $X$ is non-degenerate, the affine span of $X_1$ coincides with~$H$.
So the affine span of $k_0 X_1 \subseteq \sigma_{k_0} X_1$ equals $k_0 H$. We conclude that 
$\sigma_{k_0} X_1=k_0 H$.\bigskip

For the last statement, note that $\sigma_{k_0} X_1$ is contained in the
ordinary $k_0$-th secant variety of $X$ and that the ordinary border
rank of a vector $v\in V$ does not change whenever we multiply the vector by a non-zero constant. 
\end{proof}

\begin{proposition}\label{prop:mrank-welldef}
Let $k_0<\infty$ be the generic monic rank of elements of $V\setminus\ker(h)$. Then, 
for every vector $v \in V \setminus \ker(h)$, we have $\rk_X(v) \leq \mrk_{X,h}(v) \leq 2 k_0$. 
In particular, the monic rank of $v$ is finite.
\end{proposition}
\begin{proof}
We adopt a similar strategy as in the proof of \cite[Theorem 3]{Blekherman15}. 
Recall that the generic monic rank $k_0$ is finite by Proposition
\ref{prop:mongenexists}. Set 
$$
\tilde{v} = \frac{2k_0}{h(v)}\cdot v \in 2k_0H
$$
and consider the intersection
$$
\left(\tilde{v} - \osec_{k_0} X_1\right)\cap \osec_{k_0} X_1 \subseteq k_0 H. 
$$ 
Both sets on the left-hand side contain an open dense subset of $k_0 H$. Thus they must
intersect. Consequently, there exist $p_1,p_2\in \osec_{k_0}X_1$
such that $\tilde{v}-p_1 = p_2$. So 
$$
\frac{2k_0}{h(v)}\cdot v=\tilde{v} = p_1 + p_2\in \osec_{2k_0}X_1
$$ 
and hence $\mrk_{X,h}(v)\leq 2k_0$. This shows the second inequality. The first is immediate from
the definitions of $X$-rank and monic $X$-rank. 
\end{proof} 

One of the conditions of Theorem~\ref{thm:Basic} is that we require the closed cone $X$ to be irreducible. In the remainder of this section, we discuss the finiteness of the monic rank when this is not the case. So assume that $X^{(1)},\dots,X^{(s)}$ are the irreducible components of $X$. When some component $X^{(i)}$ does not intersect $H$, then it does not contribute to the monic rank, so we assume this is not the case. For each $i\in\{1,\dots,s\}$, we write
$$
H^{(i)}:=\span_{\KK}\left(X^{(i)}\right)\cap H=p^{(i)}+U^{(i)}
$$
with points $p^{(i)}\in H_i$ and subspaces $U^{(i)}\subseteq\ker(h)$. Furthermore, we take 
\begin{eqnarray*}
U&:=&U^{(1)}+\dots+U^{(s)}\subseteq\ker(h)\\
p&:=&(p^{(1)}+\dots+p^{(s)})/s\in H\\
\widetilde{H}&:=&\bigcup_{\substack{q_1,\dots,q_s\in\QQ_{\geq0}\\q_1+\dots+q_s=1}}q_1H^{(1)}+\dots+q_sH^{(s)}\subseteq H
\end{eqnarray*}
and note that $H=p+\ker(h)$. We now have the following theorem:

\begin{theorem}
The following are equivalent:
\begin{enumerate}
\item The set of monic $X$-ranks is bounded.
\item The monic $X$-rank of every element of $V\setminus\ker(h)$ is finite.
\item We have $H=\widetilde{H}$.
\item We have $U=\ker(h)$.
\end{enumerate}
Furthermore, when these equivalent conditions are satisfied, the maximal monic $X$-rank is at most $s$ times the product of the maximal monic $X^{(i)}$-ranks of elements in $H^{(i)}$.
\end{theorem}
\begin{proof}
$(1)\Rightarrow (2)$. When the set of monic $X$-ranks is bounded, the monic $X$-rank of every element of $V\setminus\ker(h)$ is obviously finite.\\
$(2)\Rightarrow (3)$. Note that the monic $X$-rank of every element of $V\setminus\ker(h)$ is finite if and only if for every point $v\in H$ there exist $x_1,\dots,x_k\in X\cap H$ such that $v=(x_1+\dots+x_k)/k$. So, since $X\cap H\subseteq H^{(1)}\cup\dots\cup H^{(s)}$, we see that $H=\widetilde{H}$ when the monic $X$-rank of every element of $V\setminus\ker(h)$ is finite.\\
$(3)\Rightarrow(4)$. We have $\widetilde{H}\subseteq\span_{\QQ}\left(p^{(1)},\dots,p^{(s)}\right)+U$. If $H=\widetilde{H}$, then
$$
\ker(h)=H-p\subseteq \span_{\QQ}\left(p^{(1)},\dots,p^{(s)}\right)+U
$$
and hence the images of $p^{(1)}-p^{(s)},\dots,p^{(s-1)}-p^{(s)}$ in $\ker(h)/U$ span all of $\ker(h)/U$ over~$\QQ$. Since $\KK$ has infinite dimension over $\QQ$, this is not possible when $\ker(h)/U$ has positive dimension over~$\KK$. It follows that $\ker(h)=U$.\\
$(4)\Rightarrow(1)$. If $U=\ker(h)$, then 
$$
H=p+U=1/s\cdot H^{(1)}+\dots+1/s\cdot H^{(s)}\subseteq\widetilde{H}\subseteq H.
$$ 
Let $v\in H$ be a point. Then we can write $sv=v_1+\dots+v_s$ for some $v_i\in H^{(i)}$. Using Theorem~\ref{thm:Basic}, we see that for each $i\in\{1,\dots,s\}$ there are $p_{i,1},\dots,p_{i,k_i}\in X^{(i)}\cap H^{(i)}$ such that $k_iv_i=p_{i,1}+\dots+p_{i,k_i}$. Note here that $k_i$ is at most the maximal monic $X^{(i)}$-rank of elements in $H^{(i)}$. Now, we see that
$$
s k_1\dots k_s v=\sum_{i=1}^ss k_1\dots k_s/k_i\cdot (p_{i,1}+\dots+p_{i,k_i})
$$
is the sum of $s k_1\dots k_s$ elements of $X\cap H$. So the monic $X$-rank is bounded by $s$ times the product of the maximal monic $X^{(i)}$-ranks of elements in $H^{(i)}$.
\end{proof}

\section{Invariant theory tools}\label{sec:Invariant}

Our reference for classical invariant theory is \cite{Derksen02}.

\subsection{A variant of a theorem of Hilbert}
In this section, we develop computational tools to prove Theorem~\ref{thm:Shapiro}.
Let $G$ be a reductive algebraic group over $\KK$ acting on an affine variety $Y$, whose
coordinate ring is $R:=\KK[Y]$. Let a one-dimensional torus $\TT:= \KK\setminus\{0\}$ act on $Y$. The character lattice
of $\TT$ is isomorphic to $\mathbb Z$. For any $a \in \ZZ$, let $R_a$ be the corresponding 
weight space: 
$$
R_a:=\left\{f \in \KK[Y] ~\middle|~ \forall y \in Y, t\in \TT : f(ty)=t^a f(y) \right\}. 
$$
This naturally induces a grading on $R$. Assume the following: 
\begin{enumerate}
\item[(i)] The grading $R=\bigoplus_{a \in \ZZ} R_a$ satisfies
$R_0=\KK$ and $R_a=0$ for $a<0$.
\item[(ii)] The actions of $\TT$ and $G$ on $Y$ commute. 
\end{enumerate}
Under these assumptions, each weight space $R_a$ is a representation of $G$ and
the invariant ring $R^G$ decomposes as 
$$
R^G = \mbox{$\bigoplus_{a \geq 0} R_a^G$}
$$
where $R_0^G=R_0=\KK$. In this section, the terms {\it homogeneous} and {\it degree}
refer exclusively to the grading given by $\TT$. 

\begin{proposition} \label{prop:Hilbert}
Suppose that $r_1,\ldots,r_k\in R^G$ are homogeneous elements of
positive degree such that 
$$
\mathcal V(r_1,\ldots, r_k) = \mathcal V\left(\mbox{$\bigoplus_{a>0} R_a^G$}\right)
$$
where $\mathcal{V}(S)$ denotes the vanishing set of a set of forms $S\subseteq R$. 
Then $R^G$ is a finitely generated module over its subring $\KK[r_1,\ldots,r_k]$.
\end{proposition}

\begin{remark}
In Hilbert's classical variant of this result (see, e.g., \cite[Lemma 2.4.5]{Derksen02}), the variety $Y$ is a vector space,
$G$ acts linearly on $Y$ and the grading is the standard one. The proof of our generalization is identical, but we include it for the sake of completeness. 
\end{remark}

\begin{proof}
Let $f_1,\ldots,f_{\ell} \in R^G$ be homogeneous generators of $R^G$ of positive degree. Then by Hilbert's Nullstellensatz, there exists a positive integer $m$ such
that, for each $j\in\{1,\ldots,\ell\}$, the form $f_j^{m}$ is in the ideal generated by the homogeneous forms $r_1,\ldots,r_k$. 
Using the Reynolds operator, one finds that
\[ f_i^m=\sum_{i=1}^k h_i r_i \]
for some homogeneous $h_i\in R^G$ with $\deg(h_i) < m \deg(f_i)$. So we see that the finite set
$$
\left\{\prod_{i=1}^{\ell} f_i^{m_i}~\middle|~0\leq m_1,\dots,m_{\ell}< m\right\}
$$
generates $R^G$ as a $\KK[r_1,\ldots,r_k]$-submodule of $R$.
\end{proof}

\subsection{A criterion for closedness of the open monic secant variety}
Recall that we have a non-degenerate irreducible affine cone $X \subseteq V$ and a linear function $h \in V^*\setminus\{0\}$, the affine hyperplane $H = h^{-1}(1)$ and the affine-hyperplane section $X_1 = X\cap H$. Assume a one-dimensional torus $\TT$ acts linearly on $V$ such that the following conditions are satisfied:
\begin{itemize}
\item[(1)] The action $\TT \times V \to V$ extends to a morphism
$\KK \times V \to V$ (or equivalently, all weights of $V$ are nonnegative). 
\item[(2)] The space $(V^*)^{\TT}$ of $\TT$-invariant linear forms is spanned by $h$.
\item[(3)] The set $X$ is stable under $\TT$.
\end{itemize}
Let $h_0\!\!\!:=\!\!\!h, h_1, \ldots, h_m$ be a basis of $V^*$ consisting of
$\TT$-weight vectors (i.e., a basis of homogeneous vectors). 
Since $h$ is $\TT$-invariant, the sets $H$
and $X_1=X \cap H$ are stable under~$\TT$.
The  affine section $X_1$ contains the unique $\TT$-fixed
point $x_0 \in H$. The coordinates
of $x_0$ in the given basis of $V^{*}$ are $(h_0(x_0),h_1(x_0),\ldots,h_m(x_0))=(1,0,\ldots,0)$.\bigskip

We now fix a positive integer $k$, let $Y:=X_1^k$ be the $k$-fold direct product over $\KK$ of the affine
variety $X_1$ with itself and we consider the following map:
\begin{eqnarray*}
 \phi_k\colon Y &\to& kH,\\
 (p_1,\ldots,p_k) &\mapsto& p_1+\cdots+p_k.
\end{eqnarray*}
Besides the induced action of $\TT$, the affine variety $Y$ comes naturally equipped with an action of the symmetric
group $\mathfrak S_k$ which permutes the $k$ factors. Note that the actions of $\mathfrak S_k$ and $\TT$ on $Y$ commute by definition and that $\phi_k$ is $\TT$-equivariant and $\mathfrak{S}_k$-invariant.

\begin{proposition} \label{prop:Closed}
In the setting above, if $\phi_k^{-1}(kx_0)=\{(x_0,\dots,x_0)\}$, then 
$\osec_kX_1 \subseteq kH$ is closed and of dimension $k\cdot\dim X_1$.
\end{proposition}

\begin{proof}
Let $R:= \KK[Y]$ be the coordinate ring of $Y$. Define $Y/\!\!/ \mathfrak{S}_k$ to be the affine variety whose coordinate ring 
is the invariant ring $R^{\mathfrak{S}_k}$, i.e., $Y/\!\!/ \mathfrak S_k := \textnormal{Spec}(R^{\mathfrak S_k})$.
Since $\phi_k$ is $\mathfrak S_k$-invariant, it factors through the quotient map $\pi\colon Y\rightarrow Y/\!\!/ \mathfrak S_k$. So we have a morphism $\psi$ that makes the diagram
$$\begin{tikzcd}
Y\arrow["\pi"]{r} \arrow{rd}[swap]{\phi_k}& Y/\!\!/\mathfrak S_k \arrow["\psi"]{d}\\
& kH
\end{tikzcd}$$
commutative. Note that $\im(\phi_k) = \osec_k X_1$.\bigskip

A classical result \cite[Lemma 2.3.2]{Derksen02} in invariant theory is that the quotient
map~$\pi$ is surjective. We show that the vertical map $\psi$ is closed. Since $\phi_k$ is $\TT$-equivariant, the map $\psi$ is $\TT$-equivariant as well. Let $\psi^{*}\colon \KK[kH] \to R^{\mathfrak S_k}$ be the pull-back map induced by~$\psi$ and take $r_i = \psi^{*}(h_i)$ for each $i\in\{1,\dots,m\}$. By definition, we have
$$
r_i(y)=r_i(\pi(y))=h_i(\psi(\pi(y)))=h_i(\phi_k(y))
$$
for all $y\in Y$ for every $i\in\{1,\dots,m\}$. Here we use that the pull-back map induced by $\pi$ is the inclusion $R^{\mathfrak S_k}\subseteq R$. We see that
$$
\{ (x_0,\ldots,x_0) \}\subseteq\mathcal{V}\left(\mbox{$\bigoplus_{a>0} R_a^{\mathfrak S_k}$}\right)\subseteq \mathcal{V}(r_1,\dots,r_m)=\phi_k^{-1}(kx_0)=\{ (x_0,\ldots,x_0) \}.
$$
This means that all these subsets of $Y$ coincide. So by Proposition~\ref{prop:Hilbert}, the morphism $\psi$ is finite and hence closed. This implies that the image $\osec_k X_1$ of $\phi_k$ is closed. It also implies that the dimension of $\osec_k X_1$ coincides with the one of $Y/\!\!/\mathfrak S_k$, which is $k\cdot\dim X_1$. 
\end{proof}

\section{Instances of Shapiro's conjecture}\label{sec:Shapiro}

Fix positive integers $k\leq d$ and $e$. As a first application of the monic rank, we look at a conjecture due to B. Shapiro.

\begin{conj}[{Shapiro's Conjecture, \cite[Conjecture 1.4]{LORS}}]
Every binary form of degree $d\cdot e$ can be written as the sum of $d$ $d$-th powers of forms of degree $e$.
\end{conj}

We settle this conjecture in some new cases by proving the following (stronger) conjecture: let $V = \KK[x,y]_{(de)}$ be the vector space of binary forms of degree $d\cdot e$, let 
$$
X = \left\{f^d ~\middle|~ f \in \KK[x,y]_{(e)}\right\}
$$
be the variety of $d$-th powers of forms of degree $e$, let $H$ be the affine hyperplane consisting of all forms whose coefficient at $x^{de}$ equals $1$ and take $X_1:=X \cap H$.

\begin{conj} \label{conj:MonicShapiro}
The addition map
\begin{eqnarray*}
\phi_k\colon X_1^k &\to& kH\\
(p_1,\dots,p_k)&\mapsto&p_1+\dots+p_k
\end{eqnarray*}
satisfies $\phi_k^{-1}(kx^{de})=\{(x^{de},\ldots,x^{de})\}$.
\end{conj}

\begin{proposition}\label{conjimpliesshapiro}
If Conjecture \ref{conj:MonicShapiro} holds for $(k,d,e)$, then $\osec_k X_1$ is closed
and of dimension~$k\cdot e$. In particular, Conjecture \ref{conj:MonicShapiro} for $(d,d,e)$ implies Shapiro's Conjecture for~$(d,e)$.
\end{proposition}
\begin{proof}
Let $\TT = \KK\setminus\{0\}$ act on $V$ via 
$$
t \cdot \left(\sum_{i=0}^{de} a_i x^{de-i} y^i\right) = \sum_{i=0}^{de} a_i t^i x^{de-i}y^i
$$
for all $t\in\TT$. This action has only positive weights and stabilizes $X$. 
Furthermore, the only $\TT$-invariant in $V^{*}$ is the linear function $h$ which selects the coefficient of $x^{de}$. Hence, the assumptions of Proposition~\ref{prop:Closed} are satisfied. This
implies the first statement. For the second statement, note
that every non-zero form $v\in V$ has in its $\GL_2$-orbit a
form $\tilde{v}$ with $h(\tilde{v}) \neq 0$. Assuming 
Conjecture \ref{conj:MonicShapiro}, we derive
$$
\rk_X(v)=\rk_X(\tilde{v})\leq \mrk_{X,h}(\tilde{v}) \leq d
$$
as desired.
\end{proof}

\begin{remark}\label{rmk: inductive strategy}
Let $(k,d,e)$ be a triple such that Conjecture \ref{conj:MonicShapiro} holds for $(k,d,e+1)$ and let $f_1,\dots,f_k\in\KK[x,y]_{(e)}$ be monic binary forms with $f_1^d+\dots+f_k^d=kx^{de}$. Then we have
$$
kx^{d(e+1)}=(xf_1)^d+\dots+(xf_k)^d
$$
and hence $xf_1=\dots=xf_k=x^{e+1}$ by the conjecture for
$(k,d,e+1)$. So we see that the conjecture for $(k,d,e+1)$
implies the conjecture for $(k,d,e)$. Conversely, assuming that
the conjecture holds for $(k,d,e)$, we get the following method to prove that the conjecture also holds for $(k,d,e+1)$: we have to prove that monic binary forms $f_1,\dots,f_k\in\KK[x,y]_{(e+1)}$ can only satisfy
$$
f_1^d+\dots+f_k^d=kx^{d(e+1)}
$$
when $f_1=\dots=f_k=x^{e+1}$. Suppose we have binary forms
$$
f_i=x^{e+1} + c_{i,1}x^ey + \cdots + c_{i,e+1} y^{e+1}\in\KK[x,y]_{(e+1)}
$$
satisfying the equation. If $c_{i,e+1}=0$ for each $i$, then we get
$$
\left(\frac{f_1}{x}\right)^d+\dots+\left(\frac{f_1}{x}\right)^d=kx^{de}
$$
and hence $f_1/x=\dots=f_k/x=x^e$ by the conjecture for $(k,d,e)$. Otherwise we can assume, by permuting the $f_i$ and acting with $\TT$, that $c_{1,e+1}=1$. Now, we expand the sum
$$
f_1^d+\cdots+f_k^d=kx^{d(e+1)} + r_1 x^{d(e+1)-1}y + \cdots + r_{d(e+1)}y^{d(e+1)}
$$
where the coefficients $r_{\ell}$ are polynomials in the $c_{i,j}$ and we compute the reduced Gr\"obner basis (with respect to some monomial ordering) of the ideal generated by the $r_{\ell}$ and the polynomial $c_{1,e+1}-1$. If the conjecture holds for $(k,d,e+1)$, then this Gr\"obner basis will be $\{1\}$. And, if the Gr\"obner basis is $\{1\}$, then $c_{1,e+1}$ cannot be $1$, which means that $f_1=\dots=f_k=x^{e+1}$ and hence that the conjecture holds for $(k,d,e+1)$.
\end{remark}

We can now prove Theorem~\ref{thm:Shapiro}.

\begin{proof}[Proof of Theorem~\ref{thm:Shapiro}]
We consider the cases of the theorem separately.

\begin{itemize}
\item[(i)] The coefficients of the binary form
$$
(x+a_1y)^d+\dots+(x+a_dy)^d
$$
at $x^{d-1}y,\dots,y^{d}$ are, up to constant factors, power sums in the variables $a_1,\dots,a_d$. These generate the invariant ring $\KK[a_1,\dots,a_d]^{\mathfrak{S}_d}$. This implies that 
$$
\mathcal V(r_1,\ldots,r_d) = \{(0,\ldots,0)\}
$$
and hence Conjecture \ref{conj:MonicShapiro} holds for $(d,d,1)$. So (i) holds by Proposition \ref{conjimpliesshapiro}.

\item[(ii)] Both Conjecture 2 and Shapiro's Conjecture are trivial for $d=1$. Assume that $d=2$ and let $f \in \KK[x,y]_{(2e)}$ be a binary form whose coefficient at $x^{2e}$ equals $2$. Then $f = g_1^2 + g_2^2$ for some (not necessarily monic) binary forms $g_1,g_2\in\KK[x,y]_{(e)}$ by \cite[Theorem 4]{FOS}. Fix $x^e,x^{e-1}y,\dots,y^e$ as the basis of $\KK[x,y]_{(e)}$ and consider $g_1\cdot g_1+g_2\cdot g_2\in\Sym^2(\KK[x,y]_{(e)})$ as a symmetric matrix. The linear map
\begin{eqnarray*}
\pi\colon \Sym^2(\KK[x,y]_{(e)})&\to&\KK[x,y]_{(2e)}\\
x^{e-i}y^i\cdot x^{e-j}y^j&\mapsto&x^{2e-(i+j)}y^{i+j}
\end{eqnarray*}
sends $g_1\cdot g_1+g_2\cdot g_2$ to $f$. From this we see
that the matrix $g_1\cdot g_1+g_2\cdot g_2$ has a $2$ as entry
in its top-left corner. Now it follows from Proposition \ref{prop:symmmonicrankmatrices} that
$$
g_1\cdot g_1+g_2\cdot g_2=h_1\cdot h_1+h_2\cdot h_2\in \Sym^2(\KK[x,y]_{(e)})
$$
for some monic binary forms $h_1,h_2\in\KK[x,y]_{(e)}$. So we have
$$
f=\pi(h_1\cdot h_1+h_2\cdot h_2)=h_1^2+h_2^2
$$
as desired.

\newcommand{\stackedlabel}{\raisebox{-13pt}[0pt][0pt]{\begin{tabular}{c}(iii)\\+\\(iv)\end{tabular}\!\!}}
\item[\stackedlabel] The remaining cases are checked by computer, but we use one more observation: the system of homogeneous equations in the $c_{i,j}$ constructed in the inductive strategy given in Remark \ref{rmk: inductive strategy} has only integral coefficients and is homogeneous relative to
the grading coming from the action of one-dimensional torus $\TT$. Hence, we are
checking whether a certain subvariety of a weighted projective space defined over~$\ZZ$ has no $\KK$-points. To achieve this, it is enough to 
show that the subvariety has no $\overline{\FF_p}$-points for
some prime $p$. This allows us to work modulo some prime (e.g., the prime $p=101$ is enough), which makes
the computation more efficient and lets it finish successfully. \qedhere
\end{itemize}
\end{proof}

\section{Minimal orbits}\label{sec:Minorbit}

Let $G$ be a connected and reductive algebraic group over $\KK$. Let $V$
be an irreducible rational representation of $G$. Fix a Borel subgroup $B$ of
$G$ and a maximal torus $\TT$ in~$B$. Let $v \in V$ span the unique
$B$-stable one-dimensional subspace in $V$, i.e., the highest weight space. 
Set $X:=G\cdot v \cup \{0\}$. This is the affine cone over the homogeneous variety given by the orbit of $v$. 
Let $h \in V^*$ be the function  that spans the unique $B$-stable one-dimensional subspace
of $V^*$ and is normalized so that $h(v)=1$. In this setting, we study
Question~\ref{que:Minorbit}, i.e., whether the maximal $X$-rank of a vector
in $V$ is also the maximal monic $X$-rank of a vector in $V \setminus\ker(h)$.
As positive evidence, we treat the examples from Theorem \ref{thm:Minorbit}.

\subsection{Binary forms}
Consider the case where
\begin{itemize}
	\item the vector space $V = \KK[x,y]_{(d)} \cong \Sym^d(\KK^2)$ consists of binary forms of degree $d$;
	\item the group $G = \GL_2$ acts on $V$ in the natural way;
	\item the variety $X = \{\ell^d \mid \ell\in \KK[x,y]_{(1)}\}$ consists of powers of linear forms; and
	\item the linear function $h \in V^*\setminus\{0\}$ sends a polynomial to its coefficient at $x^d$.
\end{itemize}
Here the $X$-rank is also called the {\it Waring rank}. Using
the Apolarity Lemma (see, e.g., \cite[Lemma 1.15]{IarrKan:PowerSumsBook}), one can show that $x^{d-1}y$
has Waring rank $d$ and, in fact, the maximal Waring rank of
a binary form of degree $d$ is exactly $d$ (see, e.g.,
\cite[Theorem 4.9]{RezLength}). By
Theorem~\ref{thm:Shapiro}(i), the maximal monic rank with
respect to $h$ equals $d$ as well.  Hence, the answer to
Question~\ref{que:Minorbit} is affirmative is this instance. Moreover, all
open secant varieties of~$X_1$ are closed by
Proposition~\ref{prop:Closed}---the coefficients of
$x^{d-1}y,\dots,x^{d-k}y^k$ in the sum of $k$ $k$-th powers $(x+c_i)^d$ of linear forms
are the first $k$ power sums in $c_1,\dots,c_k$ and generate the
invariant ring $\CC[c_1,\ldots,c_k]^{\mathfrak S_k}$.

\subsection{Rectangular matrices}
Consider the case where
\begin{itemize}
	\item the vector space $V = \KK^{m\times n}$ consists of $m\times n$ matrices;
	\item the group $G = \GL_m \times\GL_n$ acts by left  and right multiplication;
	\item the variety $X = \{A\in\KK^{m\times n}\mid\rk(A)\leq 1\}$ consists of rank $\leq1$ matrices; and
	\item the linear function $h \in V^*\setminus\{0\}$
	sends a matrix to its top-left entry.
\end{itemize}
Let $H \subseteq V$ be the affine space of matrices $A$ with $h(A)=1$ and take $X_1=X\cap H$.

\begin{proposition}\label{prop:monicrankmatrices}
We have 
$$
\osec_kX_1=\sigma_kX_1=\{A\in kH\mid\rk(A)\leq k\}
$$
for all $1\leq k\leq\min(m,n)$.
\end{proposition}
\begin{proof}
The inclusions
$$
\osec_kX_1\subseteq\sigma_kX_1\subseteq\{A\in kH\mid\rk(A)\leq k\}
$$
are clear. Let $A\in kH$ be a matrix with $\rk(A)\leq k$.
Our goal is to prove that $A\in \osec_kX_1$. We prove this by induction on $k$.
Write $1\leq\rk(A)=\ell\leq k$. Then, by acting with the subgroup of $h$-invariant elements of $G$,
we may assume that $A$ is the diagonal matrix with a~$k$ as its
top-left entry followed by $\ell-1$ ones. If $\ell=1$, then $A$ is the sum of $k$ copies of the matrix $E_{11}$ with just a $1$ as its top-left entry. Note that, in particular, this handles the case $k=1$. Next, assume that $k\geq\ell>1$. Then, from the fact that the equality
$$
\begin{pmatrix}k&0\\0&1\end{pmatrix}=\begin{pmatrix}k-1&\lambda\\\lambda&1-\lambda^2\end{pmatrix}+\begin{pmatrix}1&-\lambda\\-\lambda&\lambda^2\end{pmatrix},\quad \lambda=\sqrt{\frac{k-1}{k}}
$$
decomposes the matrix on the left as a sum of two matrices
of rank $1$ with $k-1$ and $1$ as entries in their top-left corners, we see that there is a decomposition $A=B+C$ with $B\in (k-1)H$ and $C\in X_1$ matrices such that $\rk(B)=\ell-1\leq k-1$. By induction, it follows that $B\in\osec_{k-1}X_1$ and hence $A\in \osec_kX_1$. This concludes the proof.
\end{proof}

It follows from the proposition that the rank of a matrix in $V$ coincides with its monic rank. So in particular, the maximal monic rank is equal to the maximal rank. 

\subsection{Symmetric matrices}\label{ssec: symmetric}
Consider the case where
\begin{itemize}
	\item the vector space 
	$$
	V= \{A\in\KK^{n\times n}\mid A=A^T\}\cong\Sym^2(\KK^n) \cong	\KK[x_1,\dots,x_n]_{(2)} 
	$$
	consists of symmetric $n\times n$ matrices;
	\item the group $G = \GL_n$ acts by $g\cdot A = gAg^T$ for $g\in G$ and $A\in V$;
	\item the variety $X=\{A\in V\mid\rk(A)\leq 1\}$ consists of rank $\leq1$ matrices; and
	\item the linear function $h \in V^*\setminus\{0\}$
	sends a matrix to its top-left entry.
\end{itemize}
Let $H \subseteq V$ be the affine space of matrices $A$ with $h(A)=1$ and take $X_1=X\cap H$.

\begin{remark}
The vector space $V$ can be viewed as the space of quadratic forms in the variables $x_1,\ldots,x_n$ by associating the quadric 
$$
(x_1,\dots,x_n)A(x_1,\dots,x_n)^T
$$
to a symmetric matrix $A$. So, the variety $X$ corresponds to the set of squares of linear forms and affine space $H$ corresponds to the set of polynomials with coefficient $1$ at $x_1^2$.
\end{remark}

\begin{proposition}\label{prop:symmmonicrankmatrices}
We have 
$$
\osec_kX_1=\sigma_kX_1=\{A\in kH\mid\rk(A)\leq k\}
$$
for all $1\leq k\leq n$.
\end{proposition}
\begin{proof}
As in the proof of Proposition \ref{prop:monicrankmatrices}, it suffices to prove that every symmetric matrix $A\in kH$ with $1<\rk(A)=\ell\leq k$ is an element of $\osec_kX_1$. We again first replace $A$ by a diagonal matrix: it is well-known that every symmetric matrix is congruent to a diagonal matrix. So we can write $A=gDg^T$ with $g\in G$ and $D\in V$ diagonal. By going though the proof of this fact, one can check that $g$ can be chosen so that its action on $V$ is $h$-invariant and $D$ is the diagonal matrix with a~$k$ as its
top-left entry followed by $\ell-1$ ones. This reduces the problem to the case where $A$ is this diagonal matrix. Now, from the fact that
$$
\begin{pmatrix}k&0\\0&1\end{pmatrix}=\begin{pmatrix}k-1&\lambda\\\lambda&1-\lambda^2\end{pmatrix}+\begin{pmatrix}1&-\lambda\\-\lambda&\lambda^2\end{pmatrix},\quad \lambda=\sqrt{\frac{k-1}{k}},
$$
we see that there is a decomposition $A=B+C$ with $B\in (k-1)H$ and $C\in X_1$ such that $\rk(B)=\ell-1\leq k-1$. We again conclude that $\osec_kX_1=\sigma_kX_1=\{A\in kH\mid\rk(A)\leq k\}$.
\end{proof}

Again, it follows from the proposition that the rank of a matrix in $V$ coincides with its monic rank. And in particular, the maximal monic rank is equal to the maximal rank. 

\subsection{\texorpdfstring{$2 \times 2 \times 2$}{2x2x2} tensors}
Consider the case where
\begin{itemize}
	\item the vector space $V = \KK^2\otimes\KK^2\otimes\KK^2$ consists of $2\times 2\times 2$ tensors;
	\item the group $G = \GL_2\times\GL_2\times\GL_2$ acts on $V$ in the natural way;
	\item the variety $X = \{v_1\otimes v_2\otimes v_3 \mid v_1,v_2,v_2\in\KK^2\}$ consists of rank $\leq1$ tensors; and
	\item the linear function $h \in V^*\setminus\{0\}$ sends a tensor to its coefficient at $e_1\otimes e_1\otimes e_1$.
\end{itemize}
Here we fix $e_1,e_2$ as a basis for $\KK^2$. Let $H \subseteq V$ be the affine space consisting of tensors $t$ with $h(t)=1$ and take $X_1=X\cap H$. Any given tensor $t\in V$ can be written as 
$$
t = M_1\otimes e_1 + M_2\otimes e_2 = 
\begin{pmatrix}
a_{11} & a_{12} \\
a_{21} & a_{22} \\
\end{pmatrix} \otimes e_1 + \begin{pmatrix}
b_{11} & b_{12} \\
b_{21} & b_{22} \\
\end{pmatrix} \otimes e_2.
$$
The matrices $M_1$ and $M_2$ in $\KK^{2\times2}$ are usually called the \textit{slices} of $t$. For ease of notation, we denote such a tensor $t\in V$ by
$$
\begin{pmatrix}a_{11}&a_{12}&\aug&b_{11}&b_{12}\\a_{21}&a_{22}&\aug&b_{21}&b_{22}\end{pmatrix}.
$$
Using this notation, we have
$$
X_1 = \left\{\begin{pmatrix}1&a&\aug&c&ac\\b&ab&\aug&bc&abc\end{pmatrix}\hspace{2pt}\middle|\hspace{3pt}a,b,c\in\KK\right\}.
$$ 
Let the group $(\KK^3,+)$ act on $V$ by
\begin{eqnarray*}
(\lambda,0,0)\cdot\begin{pmatrix}a_{11}&a_{12}&\aug&b_{11}&b_{12}\\a_{21}&a_{22}&\aug&b_{21}&b_{22}\end{pmatrix}&=&\begin{pmatrix}a_{11}&a_{12}+\lambda a_{11}&\aug&b_{11}&b_{12}+\lambda b_{11}\\a_{21}&a_{22}+\lambda a_{21}&\aug&b_{21}&b_{22}+\lambda b_{21}\end{pmatrix}\\
(0,\lambda,0)\cdot\begin{pmatrix}a_{11}&a_{12}&\aug&b_{11}&b_{12}\\a_{21}&a_{22}&\aug&b_{21}&b_{22}\end{pmatrix}&=&\begin{pmatrix}a_{11}&a_{12}&\aug&b_{11}&b_{12}\\a_{21}+\lambda a_{11}&a_{22}+\lambda a_{12}&\aug&b_{21}+\lambda b_{11}&b_{22}+\lambda b_{12}\end{pmatrix}\\
(0,0,\lambda)\cdot\begin{pmatrix}a_{11}&a_{12}&\aug&b_{11}&b_{12}\\a_{21}&a_{22}&\aug&b_{21}&b_{22}\end{pmatrix}&=&\begin{pmatrix}a_{11}&a_{12}&\aug&b_{11}+\lambda a_{11}&b_{12}+\lambda a_{12}\\a_{21}&a_{22}&\aug&b_{21}+\lambda a_{21}&b_{22}+\lambda a_{22}\end{pmatrix}
\end{eqnarray*}
for all $\lambda\in\KK$. This action is well-defined and both $X$ and $H$ are stable under it. For any $k\neq 0$, note that the $\KK^3$-orbit of any tensor $t\in kH$ contains a unique tensor of the form
$$
t'=\begin{pmatrix}
k & 0 &\aug& 0 & d_{13} \\
0 & d_{12} &\aug& d_{23} & e \\
\end{pmatrix}
$$
where $d_{12}, d_{13}, d_{23}, e\in \KK$. Indeed, modify the first slice of $t$ by using suitable elements $(u_1,0,0), (0,u_2,0)\in \KK^3$ and then the second slice by using some $(0,0,u_3)\in \KK^3$. Now, consider the second open monic secant $\osec_2 X_1$. Any tensor it contains that is of the above form must be equal to
$$\begin{pmatrix}
2 & 0 &\aug& 0 & \beta \\
0 & \gamma &\aug& \alpha & 0
\end{pmatrix}=\begin{pmatrix}
1 & a &\aug& c & ac \\
b & ab &\aug& bc & abc 
\end{pmatrix}+ \begin{pmatrix}
1 & -a &\aug& -c & ac \\
-b & ab &\aug& bc & -abc 
\end{pmatrix}$$
for some $a,b,c\in\KK$. For given $\alpha,\beta,\gamma\in\KK$, such $a,b,c$ exist unless exactly one of $\alpha,\beta,\gamma$ is equal to $0$. So 
$$
\osec_2X_1=\KK^3\cdot\left\{\begin{pmatrix}
2 & 0 &\aug& 0 & \mu_2 \\
0 & \mu_3 &\aug& \mu_1 & 0 \\
\end{pmatrix}~\middle|~\begin{array}{c}\mu_1,\mu_2,\mu_3\in\KK,\\\#\{i\mid\mu_i=0\}\neq1\end{array}\right\}
$$ 
One can verify computationally that the second monic secant $\sigma_2 X_1$ is equal to
$$
\left \lbrace  \begin{pmatrix}
2 & x_1 &\aug& x_3 & y_{13} \\
x_2 & y_{12} &\aug& y_{23} & z_{123} \\
\end{pmatrix} ~\middle|~ x_1x_2x_3+2z_{123} = x_1y_{23} + x_2y_{13}+x_3y_{12} \right\rbrace
$$
and this shows that $\osec_2X_1\neq\sigma_2X_1$.

\begin{proposition}
We have $\osec_3X_1=\sigma_3X_1=3H$.
\begin{proof}
It suffices to prove that the $\KK^3$-orbit of an arbitrary tensor $t\in 3H$ contains a tensor of the form 
$$
\begin{pmatrix}
2 & 0 &\aug& 0 & \beta \\
0 & \gamma &\aug& \alpha & 0
\end{pmatrix} + \begin{pmatrix}
1 & a &\aug& c & ac \\
b & ab &\aug& bc & abc \\
\end{pmatrix} \in \osec_3 X_1
$$
where $a,b,c\in \KK$ and $\alpha, \beta, \gamma\in \KK^{*}$. By the discussion above, the $\KK^3$-orbit of any $v\in 3H$ contains a tensor of the form
$$
\begin{pmatrix}
3 & 0 &\aug& 0 & d_{13} \\
0 & d_{12} &\aug& d_{23} & e \\
\end{pmatrix}
$$
where $d_{12}, d_{13}, d_{23}, e \in \KK$. For $a,b,c,\in \KK$, by definition of the action, we obtain
$$
(a/3,b/3,c/3)\cdot \begin{pmatrix}
3 & 0 &\aug& 0 & d_{13} \\
0 & d_{12} &\aug& d_{23} & e \\
\end{pmatrix} = \begin{pmatrix}
3 & a &\aug& c & \beta \\
b & \gamma &\aug& \alpha & \delta \\
\end{pmatrix}
$$
with
\begin{eqnarray*}
(\alpha,\beta,\gamma)&=&(d_{23} + bc/3,d_{13} + ac/3,d_{12} + ab/3),\\
\delta &=& e + \left(ad_{23}+bd_{13}+cd_{12}\right)/3 +abc/9.
\end{eqnarray*}
To finish the proof, we need to verify that there exist $a,b,c\in \KK$ such that $\alpha, \beta, \gamma\in \KK^{*}$ and $\delta = abc$ are satisfied. The condition $\alpha\in \KK^{*}$ restricted to the affine surface $\delta = abc$ is an open and dense condition. Indeed, this holds if and only if $\alpha = d_{23} + bc/3$ is not a factor of the polynomial $\delta - abc \in \KK[a,b,c]$. Perfoming the division algorithm, we find that the reminder is identically zero as polynomial in $\KK[a,b,c]$ if and only if $d_{12} = d_{13} = d_{23} = e = 0$. The same argument for $\beta$ and $\gamma$ yields the same conclusion. Thus the orbit of every 
$$
t\in 3H\setminus\left\{\begin{pmatrix}3&0&\aug&0&0\\0&0&\aug&0&0\end{pmatrix}\right\}
$$ 
contains a tensor of the claimed form. So the equality
$$
\begin{pmatrix}3&0&\aug&0&0\\0&0&\aug&0&0\end{pmatrix}=\begin{pmatrix}1&0&\aug&0&0\\0&0&\aug&0&0\end{pmatrix}+\begin{pmatrix}1&0&\aug&0&0\\0&0&\aug&0&0\end{pmatrix}+\begin{pmatrix}1&0&\aug&0&0\\0&0&\aug&0&0\end{pmatrix}
$$ 
finishes the proof. 
\end{proof}
\end{proposition}

\subsection{The adjoint representation of \texorpdfstring{$\SL_n$}{SL\_n}}
Fix a positive integer $n\in\NN$. We denote $\gl(\KK^n)$ by $\gl_n$ and $\sl(\KK^n)$ by $\sl_n$. Let $\la-,-\ra\colon\KK^n\times\KK^n\to\KK$ be the bilinear form that sends $(v,\alpha)$ to $\alpha^Tv=\tr(v\alpha^T)$ for all vectors $v,\alpha\in\KK^n$. Consider the case where 
\begin{itemize}
\item the vector space $V=\sl_n$ consists of trace-zero $n\times n$ matrices;
\item the group $G=\SL_n$ acts on $V$ by conjugation; 
\item the variety $X=\{v\alpha^T\mid v,\alpha\in\KK^n,\la v,\alpha\ra=0\}$ consists of rank $\leq 1$ matrices; and
\item the linear function $h$ sends a matrix $A$ to its top-right entry. 
\end{itemize}
Fix vectors $z,\omega\in\KK^n\setminus\{0\}$. Let $H\subseteq V$ be the affine space of matrices $A$ with $\omega^TAz=1$ and take $X_1=X\cap H$. Then we have
\begin{eqnarray*}
X_1 &=& \{v\alpha^T\mid v,\alpha\in\KK^n, \la v,\alpha\ra=0, \la v,\omega\ra\cdot\la z,\alpha\ra=1\}\\
&=& \{v\alpha^T\mid v,\alpha\in\KK^n, \la v,\alpha\ra=0, \la v,\omega\ra=\la z,\alpha\ra=1\}.
\end{eqnarray*}
In this setting, the ordinary (non-monic) open secant variety $\osec_k X$ is known to be closed and is equal to the variety of trace-zero matrices of rank $\leq k$ \cite[Theorem 1.1]{Baur04}. We write $\sigma_k X$ for this variety. The function
\begin{eqnarray*}
V&\to&\KK\\
A&\mapsto&\omega^TAz
\end{eqnarray*}
is a highest weight vector precisely when $\la z,\omega\ra=0$. If this is the case, then we can assume by changing the basis that $z=e_n$ and $\omega=e_1$, which gives us the function $h$. We will prove that every trace-zero matrix with an $n$ in its top-right corner is the sum of $n$ trace-zero matrices of rank $1$ with $1$s in their top-right corners. Our first goal is to prove the following theorem, which does not need the assumption that $\la z,\omega\ra=0$.

\begin{theorem} \label{thm:oseckX1}
The monic open secant variety $\osec_k X_1$ contains all matrices $A\in kH$ of rank~$k$ for which the following property holds:
$$
\mathscr{P}(A):\text{neither $Az$ is an eigenvector of $A$ nor $A^T\omega$ is an eigenvector of $A^T$.} 
$$
\end{theorem}

\begin{remark}
To see that the final result is {\em not} true for all $z,\omega$, note that
$$
\begin{pmatrix}2&0\\1&-2\end{pmatrix}
$$
is not the sum of two rank-one matrices of trace zero with $1$s in their top-left corners. So $z=\omega=e_1$ gives an example where the maximal monic rank is {\em not} equal to the maximal rank. However, this does not give a negative answer to Question~\ref{que:Minorbit} since the corresponding function $h$ is not a highest weight vector in this case.
\end{remark}

Note that if $\omega^TAz\neq0$, then both $Az\neq0$ and $A^T\omega\neq0$. The following lemma gives some different descriptions of the property $\mathscr{P}(A)$.

\begin{lemma}\label{lm:equivP}
Let $A\in\gl_n$ be a matrix and let $v\in\KK^n\setminus\ker(A)$ be a vector. Then the following are equivalent:
\begin{itemize}
\item[(1)] The vector $Av$ is an eigenvector of $A$.
\item[(2)] Either $v\in\ker(A^2)$ or $v-u\in\ker(A)$ for some eigenvector $u$ of $A$ with a non-zero eigenvalue.
\item[(3)] We have $Av\in\KK v+\ker(A)$.
\end{itemize}
\end{lemma}
\begin{proof}
For the implication $(1)\Rightarrow(2)$, note that 
$$
\ker(A(A-\lambda I_n))=\ker(A)\oplus\ker(A-\lambda I_n)
$$
for all $\lambda\in\KK\setminus\{0\}$. The implications $(2)\Rightarrow(3)\Rightarrow(1)$ are straightforward.
\end{proof}

Starting with an element $A\in kH \cap \sigma_kX$, we will try to find an element 
$$
J\in X_1=\{v\alpha^T\mid v,\alpha\in\KK^n, \la v,\alpha\ra=0, \la v,\omega\ra=\la z,\alpha\ra=1\}
$$ such that $A-J \in (k-1)H \cap \sigma_{k-1}X$. This explains the use of the following lemmas.

\begin{lemma} \label{lm:Step1}
Let $A\in\gl_n$ be a matrix, let $v,\alpha\in\KK^n$ be vectors and take $J=v\alpha^T$. Then the following statements hold:
\begin{enumerate}
\item $\alpha\in\im(A^T)$ if and only if $\ker(\alpha^T)\supseteq\ker(A)$.
\item $\rk(A-J)\leq\rk(A)$ if and only if $v\in\im(A)$ or $\alpha\in\im(A^T)$.
\item $\rk(A-J)<\rk(A)$ if and only if $\alpha\in\im(A^T)$ and $v=Ax$ for some vector $x\in\KK^n$ such that $\la x,\alpha\ra=1$. 
\end{enumerate}
\end{lemma}
\begin{proof}
The proofs are straightforward.
\end{proof}

\begin{lemma} \label{lm:Step2}
Let $A\in\gl_n$ be a matrix, let $x\in\KK^n\setminus\ker(A)$ be a vector and take $v=Ax \neq 0$. Then exactly one of the following is true:
\begin{itemize}
\item There exists a vector $\alpha\in\im(A^T)$ with $\la x,\alpha\ra=1$ and $\la v,\alpha\ra=0$.
\item The vector $v$ is an eigenvector of $A$ with a non-zero eigenvalue.
\end{itemize}
\end{lemma}
\begin{proof}
Note that such a vector $\alpha$ exists if and only if $x\not\in\ker(A)+\KK v$. Suppose that $x\in\ker(A)+\KK v$. Then we have $x-cv\in\ker(A)$ for some $c\in\KK\setminus\{0\}$ since $x\not\in\ker(A)$. We find that $Av=c^{-1}v$. Conversely, if $v$ is an eigenvector of $A$ with non-zero eigenvalue $\lambda\in\KK$, then $v-\lambda x\in\ker(A)$ and hence $x\in\ker(A)+\KK v$.
\end{proof}

We describe $\osec_2X_1$ in detail. In particular, we will show that this set is not closed for $n>2$ and we will also shed some light on the origin of the property $\mathscr{P}(A)$.

\begin{proposition}
A matrix $A \in 2H$ of rank $2$ lies in $\osec_2X_1$ if and only if one of
the following statements holds:
\begin{enumerate}
\item $A^2=0$;
\item $A^2\neq 0=A^3$ and $\mathscr{P}(A)$ holds;
\item $A$ is not nilpotent and $\mathscr{P}(A)$ holds; or 
\item $A$ is not nilpotent, $Az$ is an eigenvector of $A$ and $A^T\omega$ is an eigenvector of $A^T$.
\end{enumerate}
In particular, the set $\osec_2 X_1$ contains all $A\in 2H$ of rank~$2$ for which $\mathscr{P}(A)$ holds.
\end{proposition}
\begin{proof}
Suppose that $A \in 2H$ has rank $2$. We try to find an element $J=v\alpha^T\in X_1$ such that $\rk(A-J)=1$ and investigate what can go wrong. The Jordan Normal Form of $A$ can be one of the three following matrices: 
$$
\begin{pmatrix}
0\\1&0\\&&0\\&&1&0\\&&&&0\\&&&&&\ddots\\&&&&&&0
\end{pmatrix},
\begin{pmatrix}
0\\1&0\\&1&0\\&&&0\\&&&&\ddots\\&&&&&0
\end{pmatrix},
\begin{pmatrix}
\lambda\\&-\lambda\\&&0\\&&&\ddots\\&&&&0
\end{pmatrix}
$$
Let $P\in\GL_n$ be an invertible matrix and consider the following isomorphism:
\begin{eqnarray*}
\phi\colon V&\to&V\\
B&\mapsto&PBP^{-1}
\end{eqnarray*}
We see that $\phi(X)=X$. Take $\omega':=P^{-T}\omega$ and $z':=Pz$. Then one can check that replacing $(z,\omega)$ by $(z',\omega')$ causes the subset $H$ of $V$ to be replaced by its image under $\phi$. The same holds for the subsets $X_1$ and $\osec_2X_1$. This means that we are allowed to perform basechanges as long as we adjust $z$ and $\omega$ accordingly. In particular, we may assume that the matrix $A$ is in its Jordan Normal Form.\bigskip 

The first possiblity occurs precisely when $A^2=0$. In this case, we see that $Az$ is an eigenvector of $A$ and $A^T\omega$ is an eigenvector of $A^T$. Hence $\mathscr{P}(A)$ does not hold. By Lemma~\ref{lm:Step2}, for every vector $x\in\KK^n$ such that $v=Ax \neq 0$ there exists a vector $\alpha\in\im(A^T)$ with $\la x,\alpha\ra=1$. Note that $Az\in\im(A)\setminus\{0\}$ and $\la Az,\omega\ra=\omega^TAz=2$. So since $\im(A)$ is two-dimensional, we can choose the vector $x$ such that $v, Az$ are linearly independent and $\la v,\omega\ra=1$. Now, we see that $z\not\in\ker(A)+\KK x$. So we may assume that $\la z,\alpha\ra=1$. Take $J=v\alpha^T$. Then $\la v,\alpha\ra =0$ since $v\in\ker(A)$ and $\alpha\in\im(A^T)$. So $J\in\sl_n$. We also have $\omega^TJz=\la v,\omega\ra\cdot \la z,\alpha\ra =1$. So $J\in X_1$. And finally, we have $A-J\in X_1$ by Lemma~\ref{lm:Step1}.\bigskip

The second possible Jordan Normal Form occurs precisely when $A^2\neq 0=A^3$. In this case, let $v_2,v_3$ be a basis of $\im(A)$ such that $Av_2=v_3$ and $Av_3=0$. Then $Az$ is an eigenvector of $A$ if and only if $Az\in\KK v_3$. And $A^T\omega$ is an eigenvector of $A^T$ if and only if $\la v_3,\omega\ra=0$. If $A=J_1+J_2$ with $J_1,J_2\in X_1$, then $\im(J_1)\neq\KK v_3$ or $\im(J_2)\neq\KK v_3$, because the matrix $J_1+J_2$ would have rank $\leq 1$ otherwise. So $A\in\osec_2X_1$ if and only if $\rk(A-J)=1$ for some $J\in X_1$ with $\im(J)\neq\KK v_3$. For every vector $x\in\KK^n$ such that $v=Ax\not\in\KK v_3$, there exists a vector $\alpha\in\im(A^T)$ with $\la x,\alpha\ra=1$ and $\la v,\alpha\ra=0$. Since $x$, $v$ and $\ker(A)$ span $\KK^n$, this vector $\alpha$ is unique. Note that the condition $\la z,\alpha\ra=1$ is equivelent to having $Az=v+\mu Av$ for some constant $\mu\in\KK$. Write $v=b_2v_2+b_3v_3$, $Az=c_2v_2+c_3v_3$, $\la v_2,\omega\ra=a_2$ and $\la v_3,\omega\ra=a_3$. Then we get the following equivalences:
$$\begin{array}{lll}
v\not\in\KK v_3 &\mbox{if and only if}& b_2\neq0\\
\la v,\omega\ra=1 &\mbox{if and only if}& a_2b_2+a_3b_3=1\\
\la z,\alpha\ra=1 &\mbox{if and only if}& b_2=c_2 \mbox{ and } b_3+\mu b_2=c_3 \mbox{ for some } \mu\in\KK
\end{array}$$
We also have $2=\omega^TAz=\la Az,\omega\ra =a_2c_2+a_3c_3$. There is a triple $(b_2,b_3,\mu)$ satisfying the conditions on the right if and only if $c_2\neq0$ and $a_3\neq0$. And this happens precisely when $\mathscr{P}(A)$ holds. So we find that there is a $J\in X_1$ such that $A-J\in X_1$ if and only if $\mathscr{P}(A)$ holds.\bigskip

The final possible Jordan Normal Form occurs precisely when $A$ is not nilpotent. In this case $A$ is diagonalisable. Let $v_1$ be an eigenvector of $A$ with eigenvalue $\lambda\neq0$ and let $v_2$ be an eigenvector with eigenvalue~$-\lambda$. Then $Az$ is an eigenvector of $A$ if and only if $Az$ is contained in $\KK v_1\cup\KK v_2$. Let $\omega_1$ be an eigenvector of $A^T$ with eigenvalue $\lambda$ and let $\omega_2$ be an eigenvector with eigenvalue~$-\lambda$. Then $A^T\omega$ is an eigenvector of $A^T$ if and only if $A^T\omega$ is contained in $\KK \omega_1\cup\KK \omega_2$. We have $\la v_1,\omega_2\ra=\la v_2,\omega_1\ra=0$ and $\la v_1,\omega_1\ra,\la v_2,\omega_2\ra\neq0$. This means that $A^T\omega$ is an eigenvector of $A^T$ if and only if $\la v_1,\omega\ra=0$ or $\la v_2,\omega\ra=0$. Take $x\in\KK^n$ such that $v=Ax\neq0$. Then by Lemma~\ref{lm:Step2} there exists a vector $\alpha\in\im(A^T)$ such that $\la x,\alpha\ra=1$ and $\la v,\alpha\ra=0$ if and only if $v\not\in\KK v_1\cup\KK v_2$. When this is the case, note that $\KK^n$ is spanned by $x$, $v$ and $\ker(A)$ and that therefore $\alpha$ must be unique. We write $v=b_1v_1+b_2v_2$, $Az=c_1v_1+c_2v_2$, $\la v_1,\omega\ra=a_1$ and $\la v_2,\omega\ra=a_2$ and get the following equivalences:
$$\begin{array}{lll}
v\not\in\KK v_1+\KK v_2 &\mbox{if and only if}& b_1b_2\neq0\\
\la v,\omega\ra=1 &\mbox{if and only if}& a_1b_1+a_2b_2=1\\
\la z,\alpha\ra=1 &\mbox{if and only if}& b_1(1+\lambda\mu)=c_1 \mbox{ and } b_2(1-\lambda\mu)=c_2 \mbox{ for some } \mu\in\KK
\end{array}$$
We also have $a_1c_1+a_2c_2=2$. So $a_1c_1$ and $a_2c_2$ cannot both be zero. One can readily check that a solution $(b_1,b_2,c)$ for the conditions on the right exists when $a_1=c_1=0$, $a_2=c_2=0$ or $a_1a_2c_1c_2\neq0$. Conversely, assuming a solution exists, it is easy to check that $a_1=0$ if and only if $c_1=0$ and $a_2=0$ if and only if $c_2=0$. So we see that there exists a $J\in X_1$ such that $A-J\in X_1$ if and only if $a_1=c_1=0$, $a_2=c_2=0$ or $a_1a_2c_1c_2\neq0$. In the first two cases, $Az$ is an eigenvector of $A$ and $A^T\omega$ is an eigenvector of $A^T$. In the last case, the property $\mathscr{P}(A)$ holds.
\end{proof}

To reduce Theorem \ref{thm:oseckX1} to the case where $k=2$, we use the following lemma.

\begin{lemma} \label{lm:Choice}
Assume that $k\geq3$, let $W$ be a $k$-dimensional vector space, let $B\in\gl(W)$ be a matrix, let $w\in W\setminus\{0\}$ be a vector that is not an eigenvector of $B$ and let $W'$ be a subspace of $W$ complementary to $\KK w$. Then the set
$$
U=\{u\in W'\setminus\{0\}\mid B(u+w)\not\in\span(u,w)\}
$$
is open and dense in $W'$.
\end{lemma}
\begin{proof}
The complement of $U$ in $W'$ is defined by the vanishing of all $3\times 3$ minors of the matrix with columns $B(u+w)$, $u$ and $w$. Hence $U$ is open. To show that $U$ is also dense, it suffices to show that $U$ is non-empty. Let $x\in W'\setminus\{0\}$ be the projection of $Bw$ on $W'$ along $\KK w$. As $\dim(W')\geq 2$, we can choose a vector $u\in W'$ that is linearly independent of $x$. Now the affine line $B(\KK u+w)$ hits the plane $\span(u,w)$ precisely when its projection on $W'$ along $\KK w$ hits the line $\KK u$. Since $x\not\in\KK u$ lies on this projection, this happens at most once. So for almost all $t\in\KK\setminus\{0\}$, the vector $tu$ lies in $U$.
\end{proof}

\begin{proof}[Proof of Theorem \ref{thm:oseckX1} for $k\geq 3$.]
Take $k \geq 3$ and consider a matrix $A \in kH$ of rank $k$ such that neither $Az$ is an eigenvector of $A$ nor $A^T\omega$ is an eigenvector of $A^T$. We will construct a matrix $J\in X_1$ such that $A-J\in(k-1)H\cap\sigma_{k-1}X$, $(A-J)z$ is not an eigenvector of $A-J$ and $(A-J)^T\omega$ is not an eigenvector of $(A-J)^T$.\bigskip

Set $S=\{x'\in\KK^n\mid\la Ax',\omega\ra=0\}$ and $\Sigma:=\{\beta'\in\KK^n\mid\la z,A^T\beta'\ra=0\}$ and consider the affine algebraic variety $Q\subseteq(z/k+S)\times(\omega/k+\Sigma)$ defined by
$$
Q:=\{(x,\beta)\in(z/k+S) \times (\omega/k+\Sigma)\mid\la x,A^T\beta\ra=1,\la Ax,A^T\beta\ra=0\}.
$$
Any element $(x,\beta)\in(z/k+S)\times(\omega/k+\Sigma)$ satisfies
$$
\la Ax,\omega\ra=\la z,A^T\beta\ra=1
$$
and, together with the equations defining $Q$, this implies that for all $(x,\beta)\in Q$ the matrix $J:=Ax\cdot (A^T\beta)^T$ lies in $X_1$ and satisfies $\rk(A-J)<\rk(A)$. Moreover, any $J\in X_1$ with this property is of
this form by Lemma \ref{lm:Step1}.\bigskip

Now let $\pi_1$ and $\pi_2$ denote the projections from $Q$ on $z/k+S$ and $\omega/k+\Sigma$. We claim that $\pi_1(Q)$ is dense in $z/k+S$. Indeed, applying Lemma \ref{lm:Choice} where $W=\KK^n/\ker(A)$, the matrix $B$ is induced by $A$, $w=z/k+\ker(A)$ and $W'=S/\ker(A)$, we see that the set
$$
U=\{u\in W'\setminus\{0\}\mid B(u+w)\not\in\span(u,w)\}
$$
is open and dense in $W$. By taking the preimage of $U$ in $S$ and translating by $z/k$, we see therefore that the set
$$
\{x\in(z/k+S)\setminus(z/k+\ker(A))\mid Ax\not\in\span(x,z)+\ker(A)\}
$$
is open and dense in $z/k+S$. Let $x$ be an element of this set. Then the vectors $x,Ax,z$ are linearly independent modulo $\ker(A)$. So there exists a vector $\alpha\in\KK^n$ such that
$$\begin{array}{ccccc}
\la \ker(A),\alpha\ra&=&\la Ax,\alpha\ra&=&0\\
\la x,\alpha\ra&=&\la z,\alpha\ra&=&1
\end{array}$$
and we see that $\alpha=A^T\beta$ must hold for some $\beta\in\omega/k+\Sigma$ such that $(x,\beta)\in Q$. So $\pi_1(Q)$ is dense in $z/k+S$. One can similarly prove that $\pi_2(Q)$ is dense in $\omega/k+\Sigma$. This shows the existence of the matrix $J \in X_1$ such that $\rk(A-J)<\rk(A)$.\bigskip

Next we must take care of the condition that $(A-J)z$ is not an eigenvector of $A-J$ and $(A-J)^T\omega$ is not an eigenvector of $(A-J)^T$. One readily checks that for $J:=Ax\cdot (A^T\beta)^T$ with $(x,\beta)\in Q$ we have
\begin{eqnarray*}
\ker(A-J)&=&\ker(A)\oplus\KK x\\
\ker((A-J)^T)&=&\ker(A^T)\oplus\KK \beta
\end{eqnarray*}
So $(A-J)z$ is not an eigenvector of $A-J$ precisely when $(A-J)z\not\in\KK z+\KK x+\ker(A)$. And, $(A-J)^T\omega$ is not an eigenvector of $(A-J)^T$ precisely when $(A-J)^T\omega\not\in\KK\omega+\KK \beta+\ker(A^T)$.\bigskip

It is easy to meet one of these two conditions: since $\la z,A^T\beta\ra=1$, we have $Jz=Ax$.  Denote the image of a vector $u\in\KK^n$ in $\KK^n/\ker(A)$ by $\ol{u}$. Then the set
$$
\left\{x \in z/k+S ~\middle|~ \ol{x}, \ol{Ax}, \ol{z} \text{ are linearly independent},~ \ol{A(z-x)} \not \in \span(\ol{z},\ol{x})\right\}
$$
is open and dense in $z/k+S$ by two applications of Lemma \ref{lm:Choice} with $W$, $B$ and $W'$ as before and with $w\in\{z/k,-z(k-1)/k\}$. Also, as before, this set is contained in $\pi_1(Q)$. Let $Q_1$ be its pre-image in $Q$. It is tempting to do the same for $\omega$ and to claim that the two open subsets of $Q$ thus obtained must intersect. However, it is not clear whether $Q$ is an irreducible algebraic variety. So we proceed slightly more carefully.\bigskip

From now on, denote the image of the vector $\gamma\in\KK^n$ in $\KK^n/\ker(A^T)$ by $[\gamma]$ . Let $(x,\beta)$ be a point in $Q_1$. Then the tangent space $T_{(x,\beta)}Q_1$ consists of all $(x',\beta')\in S\times\Sigma$ such that 
\begin{eqnarray*}
\la x',A^T\beta\ra &=& -\la Ax,\beta'\ra\\
\la x',(A^T)^2\beta\ra &=& -\la A^2 x,\beta' \ra
\end{eqnarray*}
We claim that, if $[\beta], [A^T\beta], [\omega]$ are linearly independent, then the differential $d_{(x,\beta)} \pi_2$ is surjective. Indeed, let $\beta'\in\Sigma$. Then $(x',\beta')$ lies in $T_{(x,\beta)}Q_1$ if and only if $x'\in\KK^n$ satisfies the two equations above and $\la x',A^T\omega\ra=0$. As $A^T\beta,(A^T)^2\beta, A^T\omega$ are linearly independent, there exists a solution $x'$ to this system of linear equations. This proves the claim. We conclude that if $Q_1$ contains a point $(x,\beta)$ where $[\beta],[A^T\beta], [z]$ are linearly independent, then $\pi_2(Q_1)$ contains an open dense subset of $\omega/k+\Sigma$. This subset intersects the open dense subset where $[A^T(\omega-\beta)]\not\in\span([\omega],[\beta])$ and so we are done.\bigskip

Hence assume that all points $(x,\beta)\in Q_1$ have the following two properties: 
\begin{itemize}
\item The vectors $[\beta],[A^T\beta],[\omega]$ are linearly dependent.
\item The vector $[A^T(\omega-\beta)]$ is contained in $\span([\omega],[\beta])$.
\end{itemize} 
We derive a contradiction as follows:  for all $(x,\beta)\in Q_1$, we have 
$$
[\beta]\not\in\span\left([A^T\beta],[\omega-\beta]\right)
$$
since $\la Ax,A^T\beta\ra=\la Ax,\omega-\beta\ra=0$ and $\la Ax,\beta\ra=1$. So the linear dependence of $[\beta], [A^T\beta], [\omega]$ implies that $[A^T\beta]\in\KK[\omega-\beta]$. Together with the second assumption this implies that $[A^T\omega]\in\span([\omega],[\beta])$. Denote $\pi_2$ followed by the projection $\KK^n\to\KK^n/\ker(A^T)$ by $\varphi$. Then the map $\varphi$ sends $Q_1$ into the affine line $\ell=[\omega]/k+\KK[\gamma]$, where $\gamma$ is the projection of $A^T\omega$ on $\Sigma$ along $\KK\omega$. We claim that the differential of the map $\varphi\colon Q_1\to\ell$ is non-zero, and hence surjective, at any point $(x,\beta)\in Q_1$. Indeed, the set $\la S, A^T\beta\ra$ is not equal to $0$, because otherwise $A^T\beta$ would be a scalar multiple of $A^T\omega$ and this contradicts the linear independence of $[\beta]$ and $[\omega]$. So we may choose $x'\in S$ such that $\la x',A^T\beta\ra\neq0$. Given this $x'$, we must find a $\beta'\in\KK^n$ such that $\la x',A^T\beta\ra = -\la Ax,\beta'\ra$, $\la x',(A^T)^2\beta\ra = -\la A^2 x,\beta' \ra$ and $\la Az,\beta' \ra=0$. This is possible because $Ax, A^2x$, and $Az$ are linearly independent. The resulting $\beta'$ is clearly not an element of $\ker(A^T)$ and hence $[\beta']$ is a non-zero element of the image of $d_{(x,\beta)}\varphi$. We conclude that the image of $Q_1$ in $\ell$ contains an open dense set. On the other hand, as $[\omega]/k\in\ell$ is not an eigenvector of $A^T$ modulo $\ker(A^T)$, we have $[A^T\beta]\not\in\KK[\omega-\beta]$ for $[\beta]$ in an open dense subset of $\ell$ (the complement of which is characterised by the vanishing of $2\times2$ minors of the matrix with columns $[A^T\beta],[\omega-\beta]$). This open set must intersect $\varphi(Q_1)$, a contradiction.
\end{proof}

From now on, we take $z=e_n$ and $\omega=e_1$. So $H$ is the affine subspace consisting of all matrices $A\in V$ with a $1$ in their top-right corner and we have
$$
X_1=\{v\alpha^T\mid v,\alpha\in\KK^n, \la v,\alpha\ra=0, \la v,e_1\ra=\la e_n,\alpha\ra=1\}.
$$
We will prove the following theorem.

\begin{theorem}\label{thm:mainsln}
We have $nH=\osec_nX_1$.
\end{theorem}

Let $A\in nH$ be a matrix of rank $n$. Then $\ker(A)=0$ and we see that $Ae_n\not\in\KK e_n$ and $A^Te_1\not\in\KK e_1$ since $\la Ae_n,e_1\ra=n$. So $\mathscr{P}(A)$ holds by Lemma~\ref{lm:equivP} and $A\in\osec_nX_1$ by Theorem~\ref{thm:oseckX1}. This leaves the matrices $A\in nH$ of rank $\leq n-1$. The next three lemmas show that, roughly speaking, we can substract elements of $X_1$ from such matrices to get a matrix satisfying the property $\mathscr{P}$.

\begin{lemma} \label{lm:HalfP1}
Assume that $3 \leq k \leq n$ and let $A \in kH$ be a matrix with $2\leq\rk(A)\leq k-1$. Then
there exists a matrix $J \in X_1$ with $\rk(A-J)=\rk(A)$ such that $(A-J)e_n$ is not an eigenvector of $A-J$.
\end{lemma}
\begin{proof}
The affine space $S:=\{v\in\KK^n\mid\la v, e_1\ra=1\}$ of codimension $1$ in $\KK^n$ is not contained in 
$$
\im(A)\cup\{v\in\KK^n\mid Av\in\KK Ae_n\}\cup (Ae_n+\ker(A)+\KK e_n)
$$
since the first two of these three sets are subspaces of $\KK^n$ of codimension $\geq 1$ and the last is an affine subspace of codimension $\geq 1$ in $\KK^n$ that is not identical to $S$ (as $Ae_n\not\in S$). Let $v\in S$ be an element outside these three spaces. Then $v$, $Av$, and $Ae_n$ are linearly independent. So there is a $\beta\in\KK^n$ such that $\la v,\beta\ra=1$, $\la Av,\beta\ra=0$ and $\la Ae_n,\beta\ra=1$. Now take $\alpha=A^T\beta$ and $J=v\alpha^T$. Then we have $J\in X_1$, we have $\rk(A-J)=\rk(A)$ by Lemma~\ref{lm:Step1} and we have
$$
(A-J)e_n=Ae_n-v\not\in\ker(A)+\KK e_n=\ker(A-J)+\KK e_n
$$
as required. Here we use that $\alpha\in\im(A^T)$ and hence $\ker(J)\supseteq\ker(A)$.
\end{proof}

\begin{lemma} \label{lm:HalfP2}
Assume that $n\geq 3$ and $3\leq k\leq n-1$. Let $A \in kH$ be a matrix with $2\leq\rk(A)\leq k-1$ such that $Ae_n$ is not an eigenvector of $A$. Then there exists a matrix $J \in X_1$ with $\rk(A-J)=\rk(A)$ such that $\mathscr{P}(A-J)$ holds.
\end{lemma}
\begin{proof}
Since $Ae_n$ is not an eigenvector of $A$, we know that $Ae_n\not\in\KK e_n+\ker(A)$ by Lemma~\ref{lm:equivP}. We show that we may take $J=A(e_n/k)\alpha^T$ for some appropriate $\alpha\in\KK^n$. For such a matrix $J$ to be an element of $X_1$,
it is necessary and sufficient that $\la e_n,\alpha\ra=1$ and $\la Ae_n,\alpha\ra=0$. As the vectors $e_n$ and $Ae_n$ are linearly independent, this system of equations has a solution~$\alpha_0$. Let $\Phi\subseteq\KK^n$ be the set of solutions of the homogeneous equations $\la e_n,\alpha'\ra=0$ and $\la Ae_n,\alpha'\ra=0$.\bigskip

We have $\rk(A^T) \leq n-2$. So the affine space $\alpha_0+\Phi$ of codimension $2$ is not contained in $\im(A^T)$. Assume that $\alpha\in(\alpha_0+\Phi)\setminus\im(A^T)$. Then we have $\rk(A-J)=\rk(A)$. More specifically, we know that $\ker(A)\cap\ker(\alpha^T)$ is a proper subspace of $\ker(A)$ and that
$$
\ker(A-J)=(\ker(A)\cap\ker(\alpha^T))+\KK x
$$ 
for some $x\in\KK^n$ such that $Ax=Ae_n$. So we also have 
$$
(A-J)e_n=\frac{k-1}{k}Ae_n\not\in\ker(A)+\KK e_n=\ker(A-J)+\KK e_n
$$
and this shows that $(A-J)e_n$ is not an eigenvector of $A-J$. To make sure that $\mathscr{P}(A-J)$ holds, we need to choose $\alpha$ such that
$$
(A-J)^Te_1=A^Te_1-\alpha\not\in\ker((A-J)^T)+\KK e_1 =\ker(A^T)+\KK e_1
$$
also holds. Note here that $\ker((A-J)^T)=\ker(A^T)$ since $\im(A-J)=\im(A)$.\bigskip

So suppose on the contrary that $\alpha_0+ \Phi \subseteq A^Te_1+\ker(A^T)+\KK e_1$. Then $\Phi$ is contained in $\ker(A^T)+\KK e_1$. So any element $\alpha'\in\Phi$ may be written as $\beta+\lambda e_1$ with $\beta\in\ker(A^T)$ and $\lambda\in\KK$. But then we find that
$$
0=\la Ae_n,\alpha' \ra=\la e_n,A^T\beta \ra+\lambda\la Ae_n,e_1\ra=k\lambda
$$
and hence $\Phi$ must in fact be contained in $\ker(A^T)$. As $\Phi$ has codimension $2$ in $\KK^n$ and $\rk(A)\geq 2$, we see that $\Phi$ and $\ker(A^T)$ must be equal. So $\alpha_0-A^Te_1\in\Phi+\KK e_1$. However this is not possible, because we have $\la e_n,\alpha_0-A^Te_1\ra=1-k$ and $\la e_n,\Phi+\KK e_1\ra=0$. So $\alpha_0+ \Phi$ is not contained in $A^Te_1+\ker(A^T)+\KK e_1$. We conclude that the set
$$
(\alpha_0+\Phi)\setminus\left(\im(A^T)\cup(A^Te_1+\ker(A^T)+\KK e_1)\right)
$$
is non-empty. For any element $\alpha$ in this set, the matrix $J=A(e_n/k)\alpha^T$ has the required properties.
\end{proof}

\begin{lemma} \label{lm:nmin1}
Assume that $n \geq 3$ and let $A \in nH$ be a matrix of rank $n-1$. Then there exists a matrix $J \in X_1$ with $\rk(A-J)=n-1$ such that $\mathscr{P}(A-J)$ holds.
\end{lemma}
\begin{proof}
The proof resembles that of Theorem \ref{thm:oseckX1}, except that we have slightly more freedom in choosing the matrix $J$ since we do not need to make sure that the rank of $A$ decreases when subtracting $J$. Define
\begin{eqnarray*}
S&:=&\{x'\in\KK^n \mid \la Ax',e_1\ra=0\}\\
\Sigma&:=&\{\alpha\in\KK^n\mid\la e_n,\alpha\ra=1\}\\
Q&:=&\{(x,\alpha)\in(e_n/n+S)\times\Sigma\mid\la Ax,\alpha\ra=0,\la x,\alpha\ra=1\}
\end{eqnarray*}
If $(x,\alpha)$ is a point in $Q$, then $J=Ax\cdot\alpha^T$ lies in $X_1$ and $\im(J)\subseteq\im(A)$. So we have $\rk(A-J)\leq\rk(A)$. Note that any matrix $J$ with these properties is of this form by Lemma~\ref{lm:Step1}. Consider the set
$$
U_2:=\{\alpha\in\Sigma\mid A^T\alpha, A^Te_1\mbox{ linearly independent},\alpha\not\in\im(A^T), A^Te_1-\alpha\not\in\ker(A^T)+\KK e_1\}.
$$
The intersections $\{\alpha\in\KK^n\mid A^T\alpha\in\KK A^Te_1\}$, $\im(A^T)$ and $A^Te_1+\ker(A^T)+\KK e_1$ with $\Sigma$  are all affine subspaces of codimension $\geq 1$ in $\Sigma$. So the set $U_2$ is open and dense in~$\Sigma$.\bigskip

We claim that $U_2$ is contained in the projection $\pi_2(Q)$ of $Q$ on $\Sigma$. We also claim for any element $(x,\alpha)\in Q_2:=\pi_2^{-1}(U_2)$ and for $J=Ax\cdot\alpha^T$ that we have $\rk(A-J)=n-1$ and that $(A-J)^Te_1$ is not an eigenvector of $(A-J)^T$. This means that at least half of property $\mathscr{P}(A-J)$ holds. Indeed, if $\alpha$ is an element of $U_2$, then $\alpha, A^T\alpha, A^Te_1$ are linearly independent and therfore the system of equations 
\begin{eqnarray*}
\la x,\alpha \ra&=&1\\
\la x, A^T \alpha \ra&=&0\\
\la x, A^Te_1\ra&=&1
\end{eqnarray*}
has a solution $x\in\KK^n$. For any such $x$, the point $(x,\alpha)$ lies in $Q$. For the second claim, note that $\rk(A-J)$ cannot be lower than $\rk(A)$ since $\alpha\not\in\im(A^T)$. And from
$$
(A-J)^Te_1=A^Te_1-\alpha\not\in\ker(A^T)+\KK e_1=\ker((A-J)^T)+\KK e_1
$$
we see that the vector $(A-J)^Te_1$ is not an eigenvector of $(A-J)^T$. Here we again use that $\ker(A^T)=\ker((A-J)^T)$ since $\im(A)=\im(A-J)$.\bigskip

Since $(A-J)x=Ax-Ax\cdot \la x,\alpha\ra=0$ and $\rk(A-J)=n-1$, we have $\ker(A-J)=\KK x$. So we want to find a point $(x,\alpha) \in Q_2$ such that 
$$
A(e_n-x)=(A-J)e_n\not\in\ker(A-J)+\KK e_n=\span(x,e_n).
$$ 
The set of $x \in (e_n/n+S) \setminus \{e_n/n\}$ satisfying $A(e_n-x)\in\span(x, e_n)$ is a closed subset of $(e_n/n+S) \setminus \{e_n/n\}$. Suppose that this subset is all of $(e_n/n+S)\setminus\{e_n/n\}$. Then we see that
$$
\frac{n-1}{n}\cdot Ae_n-tAx'\in\span(x',e_n)
$$
for all $t\in\KK\setminus\{0\}$ and $x'\in S\setminus\{0\}$. From this follows that $Ae_n$ is contained in $\span(x',e_n)$ for all $x'\in S\setminus\{0\}$. Since $\dim(S)\geq2$, this can only happen if $Ae_n\in\KK e_n$, but this contradicts the fact that $\la Ae_n, e_1\ra=n\neq0$. Hence we have $A(e_n-x)\not\in\span(x,e_n)$ for all~$x$ in an open dense subset $U_1$ of $e_n/n+S$. This means that if the differential $d_{(x,\alpha)} \pi_1$ is surjective (onto $S$) in some point $(x,\alpha)\in Q_2$, then the dense image $\pi_1(Q_2)$ intersects the set $U_1$ and we are done. We claim that this is the case if $x,Ax,e_n$ are linearly independent. Indeed, $(x',\alpha') \in \KK^n \times \KK^n$ lies in $T_{(x,\alpha)} Q_2$ if and only if the following equations hold:
\begin{eqnarray*}
\la x',A^Te_1\ra&=&0\\
\la e_n,\alpha'\ra&=&0\\
\la x',A^T\alpha\ra&=&-\la Ax, \alpha' \ra\\
\la x',\alpha \ra&=&-\la x, \alpha' \ra
\end{eqnarray*}
Now if $x,Ax,e_n$ are linearly independent, then for any $x'$ satisfying the first of these equations, there exists an $\alpha'$ satisfying the other three.\bigskip

Note for $(x, \alpha) \in Q$ that the triple $x,Ax,e_n$ is linearly dependent if and only if we have $Ax\in\KK(e_n-x)$, because $e_n-x\neq0$ and $x\not\in\ker(\alpha^T)\ni Ax, e_n-x$. Hence we are left with
the case where $\pi_1$ maps $Q_2$ into the set
$$
\{x \in e_n/n+S \mid A(e_n-x) \in \span(x,e_n), Ax \in\KK (e_n-x) \}.
$$
We show that this is impossible: it follows from these equations that $Ae_n\in\span(x,e_n)$. So $\pi_1$ maps $Q_2$ into the line $\ell=e_n/n+\KK y'$ where $y'$ is the (non-zero) projection of $Ae_n$ on $S$ along $\KK e_n$. We have $Ax\not\in\KK(e_n-x)$ for $x$ in an open dense subset of $\ell$ since $e_n$ is not an eigenvector of $A$. And this subset must be disjoint from $\pi(Q_2)$ by assumption. However, one readily checks from the equations for the
tangent space above that $\pi_1$ has non-zero differential at any point, which shows that $\pi_1(Q_2)$
contains an open dense subset of $\ell$. This is a contradiction.
\end{proof}

\begin{proof}[Proof of Theorem~\ref{thm:mainsln}]
Let $A\in nH$ be a matrix. If $\rk(A)=n$, then $A\in\osec_nX_1$ by Lemma~\ref{lm:equivP} and Theorem~\ref{thm:oseckX1}. If $\rk(A)=n-1$, then $A\in\osec_nX_1$ by Lemma \ref{lm:nmin1}. If $\rk(A)=n-\ell\leq n-2$, then
there is a matrix $J_1\in X_1$ by Lemma \ref{lm:HalfP1} and there are matrices $J_2,\ldots,J_{\ell}\in
X_1$ by Lemma \ref{lm:HalfP2} such that $\rk(A-J_1-\dots-J_{\ell})=\rk(A)=n-\ell$ and $\mathscr{P}(A-J_1-\dots-J_{\ell})$ holds. So Theorem \ref{thm:oseckX1} shows that $A-J_1-\ldots-J_{\ell} \in \osec_{n-\ell}X_1$ and hence again $A \in \osec_nX_1$.
\end{proof}

\bibliographystyle{alpha}
\bibliography{monrank}

\end{document}